\documentclass[a4paper, 12pt, notitlepage]{article}
\usepackage[T1]{fontenc}
\usepackage[cp1250]{inputenc}
\usepackage[english]{babel}
\usepackage{amsfonts}
\usepackage{amsmath}
\usepackage{enumerate}
\usepackage{graphics}
\usepackage{graphicx}

\newtheorem{definition}{Definition}[section]
\newtheorem{theorem}{Theorem}[section]

\newtheorem{coro}{Corollary}[section]

\newtheorem{lemma}{Lemma}[section]
\newtheorem{assumption}{Hypothesis}[section]
\newtheorem{proposition}{Proposition}[section]

\newenvironment{proof}{\textsc{Proof.} }{\begin{flushright}$\Box$\end{flushright}}
\newenvironment{remark}{\par\vspace{0.2cm}\textbf{Remark.} }{\par\vspace{0.2cm}}

\newcommand{\xnorm}[2]{\left\|\, #1\, \right\|_{#2}}
\newcommand{\hnorm}[1]{\left|\, #1\, \right|}

\newcommand{\dual}[1]{{#1}^*}

\newcommand{\Ha}[1]{H^{#1}(\Omega)}

\newcommand{\Lp}[1]{L^{#1}(\Omega)}

\newcommand{\Ce}[1]{C([0,T];#1)}
\newcommand{\eLp}[2]{L^{#1}(0,T;#2)}
\newcommand{\Wunp}[2]{W^{#1}(0,T;#2)}

\newcommand{\Rn}[1]{\mathbb{R}^{#1}}
\newcommand{\R}{\mathbb{R}}

\title{\bf{CLOSED-LOOP CONTROL OF A REACTION-DIFFUSION SYSTEM}}
\author{Grzegorz Dudziuk\\ \\Interdisciplinary Centre for Mathematical\\and Computational Modelling\\University of Warsaw\\Pawinskiego 5a, 02-106 Warsaw, Poland\\grzegorz.dudziuk@mimuw.edu.pl}
\date{June 29, 2011}

\begin{document}
\renewcommand{\theequation}{\thesection.\arabic{equation}}

\vspace{2.7cm}

\begin{center}
\begin{large}\bf{CLOSED-LOOP CONTROL OF A REACTION-DIFFUSION SYSTEM}
\end{large}\\
\vspace{0.2cm}
\textit{(September, 2011)}
\end{center}
\vspace{0.2cm}

\begin{center}
\textsc{Grzegorz Dudziuk}\\
Interdisciplinary Centre for Mathematical and Computational Modelling\\University of Warsaw\\Pawińskiego 5a, 02-106 Warsaw, Poland\\
E-mail: grzegorz.dudziuk@mimuw.edu.pl
\end{center}

\begin{center}
\textsc{Marek Niezgódka}\\
Interdisciplinary Centre for Mathematical and Computational Modelling\\University of Warsaw\\Pawińskiego 5a, 02-106 Warsaw, Poland\\
E-mail: m.niezgodka@icm.edu.pl
\end{center}
\vspace{1cm}

\hspace*{-0.6cm}\textbf{Abstract.}
A system of a parabolic partial differential equation coupled with ordinary differential inclusions that arises from a closed-loop control problem for a thermodynamic process governed by the Allen-Cahn diffusion reaction model is studied. A feedback law for the closed-loop control is proposed and implemented in the case of a finite number of control devices located inside the process domain basing on the process dynamics observed at a finite number of measurement points. The existence of solutions to the discussed system of differential equations is proved with the use of a generalization of the Kakutani fixed point theorem.

\renewcommand{\thefootnote}{}
\footnotetext{\hspace*{-0.6cm}\textit{Acknowledgement:} \textsc{Grzegorz Dudziuk} was supported by the International Ph.D. Projects Programme of Foundation for Polish Science operated within the Innovative Economy Operational Programme 2007-20013 funded by European Regional Development Fund (Ph.D. Programme: Mathematical Methods in Natural Sciences).}
\footnotetext{\hspace*{-0.6cm}\textit{AMS Subject Classification:} 35Q79, 35Q93.}
\footnotetext{\begin{center}
\includegraphics[width=16cm]{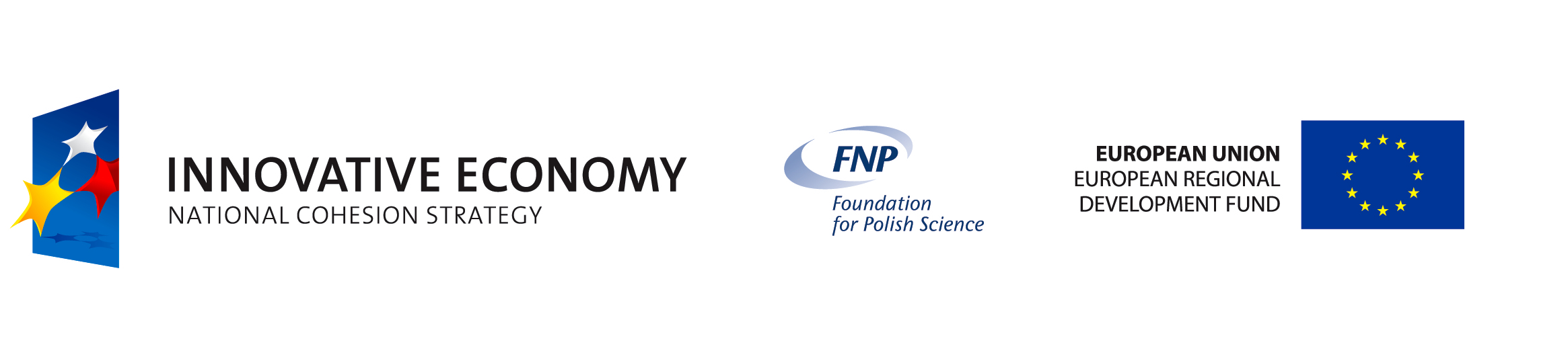}
\end{center}  }
\renewcommand{\thefootnote}{\arabic{footnote}}\setcounter{footnote}{0}

\newpage

\section{Introduction}\setcounter{equation}{0}
This paper results from a study of a control problem for a thermodynamic process governed by the reaction-diffusion equation:
\begin{equation}
 u_t(x,t) - \Delta u(x,t) = f(u(x,t))
\label{eqnrd}\end{equation}
The underlying idea refers to an implementation of the closed-loop control system with a finite number of both control units and measurement points inside the domain of the process evolution. Such a control system leads to the following system of parabolic partial differential equation coupled with ordinary differential inclusions:
\begin{equation}
\left\{\begin{array}{lrl}
u_t(x,t) - \Delta u(x,t) = f(u(x,t)) + g(x,t;\kappa(t))& \textrm{on}&Q_T\\
\beta_1 \dot{\kappa}_1(t) + \kappa_1(t)\in W_1(t,u(x_1^*,t),\ldots,u(x_n^*,t))& \textrm{on}&[0,T]\\
\vdots &\vdots&\\
\beta_m \dot{\kappa}_m(t) + \kappa_m(t)\in W_m(t,u(x_1^*,t),\ldots,u(x_n^*,t))& \textrm{on}&[0,T]\\
\quad&&\\
\frac{\partial u}{\partial n} = 0 & \textrm{on}& \partial\Omega\times (0,T)\\
u(0,x)=u_0(x) & \textrm{for}& x\in\Omega\\
\kappa_j(0)=a_j & \textrm{for}& j=1,\ldots,m
\end{array}\right.
\label{eqnmain}\end{equation}
where $Q_T=\Omega\times (0,T)$, $\Omega$ is bounded, open and sufficiently regular domain in $\Rn{d}$, $x_k^*\in\Omega$, $k=1,\ldots,n$ are fixed points in the interior of $\Omega$, $W_j$ defined on $[0,T]\times\R^n$ for $j=1,\ldots,m$ are given multivalued functions, $f:\R\rightarrow\R$ and $g:Q_T\times\R^m\rightarrow\R$. The unknown is represented by the couple $(u,\kappa)$, where $u:Q_T\rightarrow \R$, $\kappa_j:[0,T]\rightarrow \R$ and $\kappa(t)=\left(\kappa_j(t)\right)_{j=1}^m$. 

The main objective of the present paper is to explore existence questions for the solutions to system (\ref{eqnmain}).

The idea of controlling a thermodynamic process by a closed-loop control based on a finite number of control devices was already addressed in several publications. To mention some of those, problem of controlling linear heat flow was considered for example in \cite{gs1} (for one-dimensional model) and in \cite{gs2}, \cite{gj1}, \cite{gj2} (for multi-dimensional models). Closed-loop controls of the above type for the two-phase Stefan problem was also extensively investigated, cf. \cite{fh} (one-dimensional model) and \cite{hns}, \cite{cgs} (multi-dimensional case). A nonlinear parabolic system without free boundary was treated in \cite{cc} in control context, however a very special form of the nonlinear term was admitted there. 

In general, the concepts of closed-loop control systems presented in the above papers can be divided into two categories: 1) the control and measurement devices are supposed to take their actions without delay and 2) the occurrence of such delay is admitted.

The model described by (\ref{eqnmain}) corresponds to the former type. We also try to keep assumptions concerning the nonlinear term $f$ possibly general. Moreover, the structure of (\ref{eqnmain}) assumes that the control devices are placed inside the domain of the process (the term $g$ represents the control devices actions), while in all above mentioned papers the control devices were assumed to act through the boundary of the domain.

The plan of this paper is as follows. In Section~\ref{definitionofsolution} solutions to system (\ref{eqnmain}) are defined and an existence theorem for those solutions is formulated. 

Section~\ref{construction} shows an applied context for the system (\ref{eqnmain}): a construction of the closed-loop control system based on a finite number of measurement points and a finite number of control devices is presented.

Section~\ref{theproof} provides a proof of the existence theorem stated in Section~\ref{definitionofsolution}. The proof is based on a fixed point argument for a multivalued operator with the use of a generalized Kakutani fixed point theorem.

\section{Definition of solutions for system (\ref{eqnmain}) and the main existence theorem}\label{definitionofsolution}\setcounter{equation}{0}
We first formulate a set of basic notations and assumptions for this paper.

Let us denote
\begin{equation}
K=\Ce{\R^m}
\label{000eqn1.1}\end{equation}
and, for any arbitrary $R>0$,
\begin{equation}
 M_R = \left\{ (\kappa_1,\ldots,\kappa_m)\in \Wunp{1,\infty}{\R^m}:\ \xnorm{\kappa_i}{W^{1,\infty}(0,T)}\leq R,\ i=1,\ldots,m\right\}\,.
\label{000eqn1.2}\end{equation}
$M_R$ is nonempty, convex and compact subset of $K$, with compactness following via the Arzel\`a-Ascoli theorem.

We assume that:
\begin{enumerate}[(I)]
\item\label{asm1} $\Omega$ is an open and bounded domain in $\R^d$, with $C^1$-boundary.
\item\label{asm1.1} $x_k^*\in\textrm{int}\,\Omega$ for $k=1,\ldots,n$.
\item\label{asm2}
	\begin{enumerate}[a)]
	\item\label{asm2a} $f$ satisfies the linear growth condition, $f(s)s\leq c_1+ c_2 s^2$ for some positive finite constants $c_1$, $c_2$,
	\item\label{asm2b} $f$ is locally Lipschitz continuous,
	\item\label{asm2c} $f(0)=0$.
	\end{enumerate}
\item\label{asm3} The multivalued functions $W_j:[0,T]\times \R^n \longrightarrow 2^{\R}$, $j=1,\ldots,m$ satisfy the conditions:
	\begin{enumerate}[a)]
	\item\label{asm3a} $W_j(t,r)$ are nonempty, closed and convex subsets of $\R$ for every $(t,r)\in [0,T]\times \R^n$,
	\item\label{asm3b} $W_j$ are upper semicontinuous (see Definition \ref{000defap1}),
	\item\label{asm3c} $W_j$ are bounded: there exist finite $C_j>0$ such that $W(t,r)\subset [-C_j,C_j]$ for every $(t,r)\in [0,T]\times \R^n$.
	\end{enumerate}
\item\label{asm4} We set:
	\begin{enumerate}[a)]
	\item\label{asm4a} $S=\max_{j=1,\ldots,m}\left\{ \hnorm{a_j}+C_j+\frac{\hnorm{a_j}+2}{\beta_j} \right\}$,
	\item\label{asm4b} $p$ and $q$ are fixed reals such that
		\[ p>1\quad\textrm{and}\quad q\geq 2\quad\textrm{and}\quad \frac{d}{2p}+\frac{1}{q}<1\,,\]
		in particular, $q=2$ and $p=d$.
	\end{enumerate}
\item\label{asm5}
	\begin{enumerate}[a)]
	\item\label{asm5a} $g(\,.\,,\,.\,;\kappa(\,.\,))\in\eLp{2}{\Lp{\infty}}$ for every $\kappa\in K$,
	\item\label{asm5b} $\xnorm{g(\,.\,,\,.\,;\kappa^1(\,.\,)) - g(\,.\,,\,.\,;\kappa^2(\,.\,))}{L^1(Q_T)} \leq c_3 \sum_{j=1}^{m}\xnorm{\kappa^1_j - \kappa^2_j}{C[0,T]}$ \\with some positive constant $c_3$, for all $\kappa^1,\kappa^2\in M_S$,
	\item\label{asm5c} $\xnorm{g(\,.\,,\,.\,;\kappa(\,.\,))}{\eLp{q}{\Lp{p}}}\leq c_4$ \\with some positive constant $c_4$, for every $\kappa\in M_S$.
	\end{enumerate}
\item\label{asm6} $u_0\in\Lp{\infty}$.
\item\label{asm7} $\beta_j > 0$, $a_j\in\R$.
\end{enumerate}

Prior to defining solutions to the system (\ref{eqnmain}), we shall define those of the following parabolic problem:
\begin{equation}
\left\{\begin{array}{lrl}
u_t(x,t) - \Delta u(x,t) = f(u(x,t)) + g(x,t)& \textrm{on}&Q_T\\
\frac{\partial u}{\partial n} = 0 & \textrm{on}& \partial\Omega\times (0,T)\\
u(0,x)=u_0(x) & \textrm{for}& x\in\Omega
\end{array}\right.
\label{000eqn1}\end{equation}
and of the ordinary differential inclusion
\begin{equation}
\left\{\begin{array}{lrl}
\beta \dot{\gamma}(t) + \gamma(t)\in \mathbb{W}(t)& \textrm{on}&[0,T]\\
\gamma(0)=\gamma_0 & & 
\end{array}\right.\,,
\label{000eqn2}\end{equation}
the latter representing a generalization of
\begin{equation}
\left\{\begin{array}{lrl}
\beta \dot{\gamma}(t) + \gamma(t)=\mathbb{V}(t)& \textrm{on}&[0,T]\\
\gamma(0)=\gamma_0 & & 
\end{array}\right. \,.
\label{000eqn3}\end{equation}
For the system (\ref{000eqn1}) we assume the following definition:
\begin{definition}\footnote{$\left<\,.\,,\,.\,\right>$ denotes the paring between $\Ha{1}$ and its dual space, $(\,.\,,\,.\,)$ denotes the standard scalar product in $\Lp{2}$.}
For $u_0\in\Lp{2}$ and $g\in \eLp{2}{\dual{\Ha{1}}}$, $u$ is a weak solution to the problem (\ref{000eqn1}) if
\begin{displaymath}
u\in \eLp{2}{\Ha{1}}\cap\Ce{\Lp{2}},\quad u'\in \eLp{2}{\dual{\Ha{1}}}
\end{displaymath}
and
\begin{displaymath}
\int_0^T\left<\phi,u' \right> - (\phi,f(u)+g) + (\nabla \phi,\nabla u)\,dt=0
\end{displaymath}
for all $\phi \in \eLp{2}{\Ha{1}}$ and $u(0) = u_0$.
\label{000def1}\end{definition}
For $g$ is of the form $g(\,.\,,\,.\,;\kappa(\,.\,))$,  meeting the assumption (\ref{asm5}) and $\kappa\in K$, the problem (\ref{000eqn1}) admits a unique weak solution $u$ (see Theorem \ref{th1}) which is continuous in $Q_T$ (see Theorem \ref{th3}), thus a (single-valued) mapping
\begin{equation}
\mathfrak{u}:K\rightarrow C(Q_T)
\label{000eqn1.3}\end{equation}
can be defined, which to any given $\kappa\in K$ assigns the weak solution $u$ of (\ref{000eqn1}).

The solutions of (\ref{000eqn2}) are defined by the following
\begin{definition} For $\gamma_0\in\R$ and $\mathbb{W}\in L^{\infty}(0,T)$:
\begin{enumerate}[a)]
\item $\gamma$ is a solution to (\ref{000eqn2}), if in the Carath\'eodory sense it is a solution to (\ref{000eqn3}) for some integrable function $\mathbb{V}:[0,T]\longrightarrow \R$ such that $\mathbb{V}(t)\in \mathbb{W}(t)$ for a.e. $t\in [0,T]$.
\item $\gamma$ is a solution to (\ref{000eqn3}) in the Carath\'eodory sense, if $\gamma\in AC[0,T]$, $\gamma(0)=\gamma_0$ and the first equation in (\ref{000eqn3}) is fulfilled by $\gamma$ a.e. on $[0,T]$.
\end{enumerate}
\label{000def2}\end{definition}
\begin{remark}
The above definition remains correct also for a wider class of functions $\mathbb{W}$, still since in assumption (\ref{asm3c}) we put a boundedness constraint for $W_j$ appearing in (\ref{eqnmain}) it appears convenient to focus on the case of $\mathbb{W}\in L^{\infty}(0,T)$.
\end{remark}

If given functions $W_j$, $j=1,\ldots,m$ satisfy the assumption (\ref{asm3}) and $u\in C(Q_T)$, then for $\mathbb{W}_j=W_j(t,u(x_1^*,t),\ldots,u(x_n^*,t))$  entering into the right-hand side of (\ref{000eqn2}) there exists a solution in the sense of Definition \ref{000def2} (see Theorem \ref{th5}) what allows us to define, for every $\mathbb{W}_j$, $j=1,\ldots,m$, a multivalued mapping
\begin{equation}
\mathfrak{t}_j:\ C(Q_T)\ \longrightarrow\ 2^{AC[0,T]}
\label{000eqn1.4}\end{equation}
which assigns the set of solutions of (\ref{000eqn2}) to given $u\in C(Q_T)$.

\begin{proposition}
Let $\mathbb{W}$ be given such that a solution to the equation (\ref{000eqn2}) (in the sense of Definition \ref{000def2}) exists. Then $\gamma$ is a solution to (\ref{000eqn2}) if and only if
\[ \gamma(t) = e^{-\frac{1}{\beta}t}\gamma(0)\ +\ \frac{1}{\beta}\int_0^t e^{-\frac{1}{\beta}(t-s) }\,\mathbb{V}(s)\,ds\]
for some integrable function $\mathbb{V}:[0,T]\longrightarrow R$ such that $\mathbb{V}(t)\in \mathbb{W}(t)$ for a.e. $t\in [0,T]$.
\label{000le1}\end{proposition}

We are now ready for formulating the following definition:
\begin{definition}
Suppose that $\Omega$, $f$, $g$, $W_j$, $u_0$, $a_j$ and $\beta_j$ satisfy assumptions (\ref{asm1}) - (\ref{asm7}). We say that a couple $(u,\kappa)\in C(Q_T)\times [AC[0,T]]^m$ is a solution to system (\ref{eqnmain}) if and only if
\begin{enumerate}[a)]
\item $u=\mathfrak{u}(\kappa)$, i.e. $u$ is a weak solution of system (\ref{000eqn1}) in the sense of Definition~\ref{000def1}, corresponding to $\kappa$,
\item $\kappa_j\in\mathfrak{t}_j(u)$ for $j=1,\ldots,m$, i.e. $\kappa_j$ is a solution in the sense of Definition~\ref{000def2} of (\ref{000eqn2}) with $W_j(t,u(x_1^*,t),\ldots,u(x_p^*,t))$ entering into the right-hand side, or equivalently (by Proposition \ref{000le1})
\[ \kappa_j(t) = e^{-\frac{1}{\beta_j}t}\kappa_j(0)\ +\ \frac{1}{\beta_j}\int_0^t e^{-\frac{1}{\beta_j}(t-s) }\,v_j(s)\,ds\]
for some measurable function $v_j:[0,T]\longrightarrow R$, such that \[v_j(t)\in W_j(t,u(x_1^*,t),\ldots,u(x_p^*,t))\] for a.e. $t\in [0,T]$ for $j=1,\ldots,m$.
\end{enumerate}
\label{000def3}\end{definition}

We conclude this section with the main result of our paper:
\begin{theorem}
Suppose that assumptions as stated in Definition \ref{000def3} are fulfilled. Then there exists a solution to system (\ref{eqnmain}) in sense of this definition.
\label{thmain}\end{theorem}
The proof is given in Section~\ref{theproof}.

\section{Construction of the closed-loop control system}\label{construction}\setcounter{equation}{0}
As mentioned in Introduction, we assume that the control of the diffusion-reaction process (\ref{eqnmain}) is performed via a finite number of control devices localized in the interior of the process domain. The control action is based on observation data at finite set of measurement units as compared to given reference patterns.

Suppose that a reaction-diffusion process governed by (\ref{eqnrd}) with homogeneous Neumann boundary conditions and prescribed initial condition evolves over $\Omega$. The control action for this process is then described by an additive \emph{control term} (denote it by $g$) to the reaction-diffusion equation assuming so the form (\ref{000eqn1}).

Let the \emph{control term} be set as $g=g(x,t,\xi)$, where $\xi\in\R^m$ is an $m$-dimensional parameter. Moreover if, for any admissible function $\kappa:[0,T]\longrightarrow \R^m$, the term $g(x,t,\kappa(t))$ is sufficiently regular, then any such $\kappa$ can be considered as a \emph{control} applied to the system (\ref{000eqn1}) and it determines a unique solution $u$ which represents the \emph{response} of the system to $\kappa$.

To be more specific, it is our aim to keep the state of the system ,,as close as possible'' to a prescribed function $\tilde{u}:\Omega\longrightarrow \R$ which assumes values $u_k^* := \tilde{u}(x_k^*)$ at the measurement points.

Let us expose more on the above idea of measurement points and control devices. The control $\kappa$ shall be synthesized by $m$ control devices located in $\textrm{int}\,\Omega$ which have an ability to drive the thermodynamic process in the entire domain: each of $m$ components of $\kappa$ is considered to be an input signal for one of those $m$ devices. We implement the above concept by assuming the control term $g$ in (\ref{000eqn1}) to admit the representation
\begin{equation}
g(x,t,\kappa(t)) = \sum_{j=1}^m g_j(x,t)\kappa_j(t)\quad \textrm{for $(x,t)\in Q_T$}\,,
\label{000eqn5.1}\end{equation}
where $g_j$ are prescribed nonnegative functions on $Q_T$. Each of the functions $g_j$ describes a control device, $\kappa=(\kappa_1,\ldots,\kappa_m)$ and $\kappa_j$ is the signal which specifies actions of $j$-th device. $\kappa_j$ are described as the solutions of
\begin{equation}
\left\{\begin{array}{lrl}
\beta_j \dot{\kappa}_j(t) + \kappa_j(t)=v_j(t)& \textrm{on}&[0,T]\\
\kappa_j(0)=a_j & & 
\end{array}\right.
\label{000eqn5.2}\end{equation}
for $j=1,\ldots,m$, where $\beta_j$ are positive constants characterizing the $j$-th device and $v_j:[0,T]\longrightarrow\R$ can be considered as an external power supply term for $j$-th device. 

Now let us assume that for controlling the behaviour of system (\ref{000eqn1}) we use a control $\kappa$ which itself depends on a solution $u$ of (\ref{000eqn1}). To be more specific, we postulate that the latter dependence is included in the power supplies $v_j$ and admits the representation
\begin{equation}
v_j(t) = \sum_{k=1}^n \alpha_{jk}(t) w_k(u(x_k^*,t) - u_k^*)\quad \textrm{on $[0,T]$}
\label{000eqn5.3}\end{equation}
for $j=1,\ldots,m$, $k=1,\ldots,n$, where $\alpha_{jk}$ satisfies:
\begin{equation}
\begin{array}{l}
\alpha_{jk}\in C[0,T]\\
\alpha_{jk}(t)\geq 0\quad\textrm{on $[0,T]$}\\
\sum_{k=1}^n \alpha_{jk}(t)=1\quad\textrm{on $[0,T]$ for $j=1,\ldots,m$}
\end{array} \,.
\label{000eqn5.4}\end{equation}
Here $v_j$ are convex combinations of $w_k(u(x_k^*,t) - u_k^*))$ with $t$-dependent weights. It remains to define the functions $w_k$. Through many concepts concerning $w_k$, the simplest possible idea, also of practical relevance, is alternating their values between $-1$ and $+1$, depending on the sign of $u-u_k^*$ at the control units $x_k^*$:
\begin{equation}
w_k(r)=
\left\{\begin{array}{rl}
-1& \quad\textrm{for $r>0$}\\
0& \quad\textrm{for $r=0$}\\
+1& \quad\textrm{for $r<0$}
\end{array}\right.\label{000eqn5.5}
\end{equation}
Taking into account (\ref{000eqn5.1}) - (\ref{000eqn5.5}), we obtain a new system of differential equations with the unknown the couple $(u,\kappa)$:
\begin{equation}
\left\{\begin{array}{lrl}
u_t(x,t) - \Delta u(x,t) = f(u(x,t)) + \sum_{j=1}^m g_j(x,t)\kappa_j(t)& \textrm{on}&Q_T\\
\beta_1 \dot{\kappa}_1(t) + \kappa_1(t)=\sum_{k=1}^n \alpha_{1k}(t) w_k(u(x_k^*,t) - u_k^*)& \textrm{on}&[0,T]\\
\vdots &\vdots&\\
\beta_m \dot{\kappa}_m(t) + \kappa_m(t)=\sum_{k=1}^n \alpha_{mk}(t) w_k(u(x_k^*,t) - u_k^*)& \textrm{on}&[0,T]
\end{array}\right.\,,
\label{000eqnexample1}
\end{equation}
subject to suitable boundary and initial conditions. This is an algebraic system, in the form resembling (\ref{eqnmain}). Still, the proof of Theorem \ref{thmain} presented in Section~\ref{theproof} shows that assumption (\ref{asm3b}) on the upper semicontinuity of $W_j$ entering into (\ref{eqnmain}) is essential. Hence our existence result does not apply directly to (\ref{000eqnexample1}) where $\sum_{k=1}^n \alpha_{jk}(t) w_k(u(x_k^*,t) - u_k^*)$ stands for $W_j$ and $w_k$ are as in (\ref{000eqn5.5}). This suggests an idea of replacing $w_k$ with their multivalued convexifications $\tilde{w}_k$ which are upper semicontinuous:
\begin{displaymath}
\tilde{w}_k(r)=
\left\{\begin{array}{rl}
-1& \quad\textrm{for $r>0$}\\
\left[ -1,+1\right]& \quad\textrm{for $r=0$}\\
+1& \quad\textrm{for $r<0$}
\end{array}\right.
\end{displaymath}
and, accordingly, replacing the ordinary differential equations in (\ref{000eqnexample1}) by the differential inclusions:
\begin{displaymath}
\beta_j \dot{\kappa}_j(t) + \kappa_j(t)\in\sum_{k=1}^n \alpha_{jk}(t)\tilde{w}_k(u(x_k^*,t) - u_k^*)
\end{displaymath}
on $[0,T]$ for $j=1,\ldots,m$. 

Now note that assumption (\ref{asm3}) is fulfilled for $W_j=\sum_{k=1}^n \alpha_{jk}(t)\tilde{w}_k(u(x_k^*,t) - u_k^*)$ and, after ensuring that the rest of assumptions (\ref{asm1}) - (\ref{asm7}) is satisfied, including the choice of $g_j$ such that $g$ satisfies assumption (\ref{asm5}), Theorem \ref{thmain} becomes applicable. This justifies the use of the differential inclusions in (\ref{eqnmain}), since it seems natural to expect that the case of the switching control law inscribed in $W_j$ will be taken into account.

\section{The proof of the existence theorem}\label{theproof}\setcounter{equation}{0}
Let us now proceed to the proof of Theorem \ref{thmain}. The proof will exploit a fixed point argument. More precisely, it will be shown that if a certain multivalued operator on a Banach space has a fixed point (see Definition \ref{000defap3}), then this fixed point determines a solution of problem (\ref{eqnmain}) in the sense of Definition \ref{000def3} (see Subsection \ref{proof1}). Next we will show that the discussed operator indeed admits at least one fixed point (see Subsections \ref{proof2} and \ref{proof3}). 

\subsection{Construction of the fixed-point operator}\label{proof1}
This subsection is devoted to constructing the suitable multivalued operator whose fixed points determine the solutions of problem (\ref{eqnmain}) in the sense of Definition \ref{000def3}.

Let $K$ be defined as in (\ref{000eqn1.1}), $M_R$ be defined as in (\ref{000eqn1.2}), $x_k^*$, $k=1,\ldots,n$, meet assumption (\ref{asm1.1}) and $S$ fit assumption (\ref{asm4a}). For every $\xi\in M_S$, the solution $u=\mathfrak{u}(\xi)$ is continuous on $Q_T$ (where $\mathfrak{u}$ is the operator defined in (\ref{000eqn1.3})), hence the functions $u(x_k^*,\,.\,)$, $k=1,\ldots,n$ are well-defined continuous functions on $(0,T)$. Therefore the operator
\begin{eqnarray}
&&\mathbf{R}:M_S \longrightarrow \Ce{\R^n}\nonumber\\
&&\mathbf{R}(\xi) = \left( \mathfrak{u}(\xi)(x_1^*,\,.\,),\ldots,\mathfrak{u}(\xi)(x_n^*,\,.\,)\right)\nonumber
\end{eqnarray}
is well-defined. Referring to the ideas introduced in Section~\ref{construction}, operator $\mathbf{R}$ incorporates the whole information about the process at given control points. 

Denote $W=(W_1,\ldots,W_m)$. We define a multivalued operator $\mathbf{Q}:\Ce{\R^n}\longrightarrow 2^{\eLp{\infty}{\R^m}}$ as follows:
\begin{eqnarray}
&&\textrm{for given $f\in \Ce{\R^n}$, any $v\in\eLp{\infty}{\R^m}$ satisfies $v\in\mathbf{Q}(f)$}\nonumber\\
&&\textrm{if and only if $v(t)\in W(t,f(t))$ a.e. on $(0,T)$}\,.\nonumber
\end{eqnarray}
By assumption (\ref{asm3c}), each $W_j$ is bounded with some finite constant $C_j>0$. Every $v=(v_1,\ldots,v_m)\in\eLp{\infty}{\R^m}$ that belongs to $\bigcup Im(\mathbf{Q})$ satisfies $\xnorm{v_j}{L^\infty(0,T)}\leq C_j$ for $j=1,\ldots,m$. Denote $C=(C_1,\ldots,C_m)$ and
\[ \widetilde{M}_C = \left\{v= (v_1,\ldots,v_m)\in\eLp{\infty}{\R^m}:\ \xnorm{v_j}{L^\infty(0,T)}\leq C_j\ \textrm{for $j=1,\ldots,m$}\right\} \,.\]
The set $\widetilde{M}_C$ is bounded in $\eLp{\infty}{\R^m}$. As in Section~\ref{construction}, the operator $\mathbf{Q}$ determines a control law which assigns the power to be supplied for the control devices to information about the process collected in measurement points $x_k^*$, $k=1,\ldots,n$.

The values $v_j$ determine responses of the control devices $\kappa_j$, $j=1,\ldots,m$ described by ordinary differential equations 
\begin{displaymath}
\left\{\begin{array}{lrl}
\beta_j \dot{\kappa}_j(t) + \kappa_j(t)=v_j(t)& \textrm{on}&[0,T]\\
\kappa_j(0)=a_j & & 
\end{array}\right. \,,
\end{displaymath}
the latter having solutions (in the Carath\'eodory sense) given by the integral formula introduced in Proposition \ref{000le1} with $v_j$ instead of $\mathbb{V}$: 
\begin{equation}
\kappa_j(t)\ =\ e^{-\frac{1}{\beta_j}t}a_j\ +\ \frac{1}{\beta_j}e^{-t/\beta_j}\int_0^t e^{s/\beta_j}\,v_j(s)\,ds \,.
\label{000eqn2.1}\end{equation}
This formula sets continuous functions well-defined on $[0,T]$ for any $v_j\in\eLp{\infty}{\R^m}$; thus we can define an operator $\mathbf{P}:\eLp{\infty}{\R^q}\longrightarrow K$ with $\left(\mathbf{P}(v)\right)_j$ given by (\ref{000eqn2.1}) for $j=1,\ldots,m$. As in Section~\ref{construction}, the operator $\mathbf{P}$ assigns an input signal for control devices corresponding to the supplied amounts of power. 

Moreover, the solutions $\kappa_j$ are absolutely continuous and bounded, with bounded derivatives, hence $\kappa_j\in W^{1,\infty}(0,T)$. We can also conclude the following estimates for the values of $\kappa_j$ and $\dot{\kappa}_j$ under assumption that $v\in \widetilde{M}_C$:
\begin{eqnarray*}
\kappa_j(t) &=& e^{-\frac{1}{\beta_j}t}a_j + \frac{1}{\beta_j}e^{-t/\beta_j}\int_0^t e^{s/\beta_j}\,v_j(s)\,ds\ \leq\ \hnorm{a_j} + \frac{C_j}{\beta_j}e^{-t/\beta_j}\int_0^t e^{s/\beta_j}\,ds \\
&=& \hnorm{a_j} + \frac{C_j}{\beta_j}e^{-t/\beta_j} \left[ \beta_j e^{s/\beta_j}\right]_{s=0}^t\ =\ \hnorm{a_j} + \frac{C_j}{\beta_j}e^{-t/\beta_j} \left( \beta_j e^{t/\beta_j} - \beta_j\right)\\
&=& \hnorm{a_j} + C_j\left(1 - e^{-t/\beta_j}\right)\ \leq\ \hnorm{a_j} + C_j
\end{eqnarray*}
for $j=1,\ldots,m$. In the same way we show that $\kappa_j(t)\geq -\hnorm{a_j}-C_j$. For estimating $\dot{\kappa}_j$, $j=1,\ldots,m$ we proceed as follows:
\[-C_j\ \leq\quad \beta_j \dot{\kappa}_j + \kappa_j = v_j\quad \leq\ C_j \,,\]
thus
\[ \dot{\kappa}_j\ \leq\ \frac{C_j}{\beta_j} - \frac{\kappa_j}{\beta_j}\ \leq\ \frac{\hnorm{a_j}+2 C_j}{\beta_j}\]
since $\kappa_j(t)\in \left[-\hnorm{a_j} - C_j,\hnorm{a_j} + C_j\right]$. Similarly we obtain $\dot{\kappa}_j \geq -\frac{\hnorm{a_j}+2 C_j}{\beta_j}$. Altogether for all $j=1,\ldots,m$ we have
\[ \xnorm{\kappa_j}{W^{1,\infty}(0,T)}\ \leq\ \max_{j=1,\ldots,q}\left\{\hnorm{a_j} + C_j + \frac{\hnorm{a_j}+2}{\beta_j}\right\} \,.\]
\begin{coro}
For every $v\in \widetilde{M}_C$, $\mathbf{P}(v)\in M_S$ where $S$ is as in the assumption (\ref{asm4a}).
\end{coro}
Actually, the above corollary justifies the choice of the constant $S$ introduced in assumption (\ref{asm4a}).

To conclude, we have defined the following operators:
\begin{equation}
\begin{array}{lll}&&\mathbf{R}:M_S \longrightarrow \Ce{\R^n}\\
&&\mathbf{Q}:\Ce{\R^n}\longrightarrow 2^{\eLp{\infty}{\R^m}}\\
&&\mathbf{P}:\eLp{\infty}{\R^m}\longrightarrow K
\end{array}
\label{000eqn2.2}\end{equation}
such that
\begin{displaymath}
\begin{array}{lll}&&\bigcup Im(\mathbf{Q})\subseteq \widetilde{M}_C\\
&&\mathbf{P}|_{\widetilde{M}_C}:\widetilde{M}_C\longrightarrow M_S\,,
\end{array}
\end{displaymath}
thus the superposition $\mathbf{P}\circ\mathbf{Q}\circ\mathbf{R}:M_S\longrightarrow 2^{M_S}$ is a well-defined multivalued operator.

Note that by definitions of the operators $\mathbf{P}$, $\mathbf{Q}$ and $\mathbf{R}$, the condition $\kappa\in\mathbf{P}\circ\mathbf{Q}\circ\mathbf{R}(\xi)$ can be rewritten to the form:
\[ \kappa_j(t) = e^{-\frac{1}{\beta_j}t}\kappa_j(0)\ +\ \frac{1}{\beta_j}\int_0^t e^{-\frac{1}{\beta_j}(t-s) }\,v_j(s)\,ds\qquad\textrm{for}\ j=1,\ldots,m\]
for some measurable function $v_j:[0,T]\longrightarrow R$ such that $v_j(t)\in W_j(t,\mathbf{R}(\xi))$ for a.e. $t\in [0,T]$, where $W_j(t,\mathbf{R}(\xi))=W_j(t,u(x_1^*,t),\ldots,u(x_n^*,t))$ on $[0,T]$ for $u=\mathfrak{u}(\xi)$. Equivalently, in more compact form: $\kappa_j\in\mathfrak{t}_j(u)$ for $u=\mathfrak{u}(\xi)$ and $j=1,\ldots,m$ (where $\mathfrak{t}_j$ are the operators defined in (\ref{000eqn1.4})).

This brings us to the conclusion that for $\xi=\kappa$ the condition $\kappa\in\mathbf{P}\circ\mathbf{Q}\circ\mathbf{R}(\xi)$ is equivalent to the hypothesis included in Definition \ref{000def3}. Thus, we can reformulate the definition of solutions to system (\ref{eqnmain}) by introducting the following equivalent:
\begin{definition}
Suppose that $\Omega$, $f$, $g$, $W_j$, $u_0$, $a_j$ and $\beta_j$ are obeying general assumptions (\ref{asm1}) - (\ref{asm7}). We say that a couple $(u,\kappa)\in C(Q_T)\times AC([0,T];\R^m)$ is called a solution to system (\ref{eqnmain}) if and only if
\begin{enumerate}[a)]
\item $u=\mathfrak{u}(\kappa)$, i.e. $u$ is a weak solution of system (\ref{000eqn1}) in the sense of Definition~\ref{000def1}, corresponding to $\kappa$,
\item $\kappa\in\mathbf{P}\circ\mathbf{Q}\circ\mathbf{R}(\kappa)$.
\end{enumerate}
\label{000def4}\end{definition}
The latter definition allows us to conclude that if $\kappa$ is a fixed point of the operator $\mathbf{P}\circ\mathbf{Q}\circ\mathbf{R}$, i.e. $\kappa\in\mathbf{P}\circ\mathbf{Q}\circ\mathbf{R}(\kappa)$, then the couple $(\mathfrak{u}(\kappa),\kappa)$ is a solution of (\ref{eqnmain}), hence the existence of solutions (\ref{eqnmain}) will be assured once we prove the existence of at least one fixed point of $\mathbf{P}\circ\mathbf{Q}\circ\mathbf{R}$. The existence of the fixed points will be shown in the next two subsections. 

\begin{remark}
By the construction of operator $\mathbf{P}\circ\mathbf{Q}\circ\mathbf{R}$ we conclude that solutions $(u,\kappa)$ to system (\ref{eqnmain}) belong not only to $C(Q_T)\times AC([0,T];\R^m)$ but also to $C(Q_T)\times \Wunp{1,\infty}{\R^m}$.
\end{remark}

\subsection{The fixed-point argument}\label{proof2}
As already mentioned, we shall prove the existence of fixed points for the operator $\mathbf{P}\circ\mathbf{Q}\circ\mathbf{R}$ defined in Subsection \ref{proof1}. To this purpose, we will employ a certain generalization of the Kakutani fixed point theorem --- we will take three abstract operators $\mathcal{P}$, $\mathcal{Q}$ and $\mathcal{R}$ (not to be confused with $\mathbf{P}$, $\mathbf{Q}$ and $\mathbf{R}$) acting between spaces given in (\ref{000eqn2.2}) and investigate their properties that facilitate an appropriate use of the generalized Kakutani theorem.

Once we give an answer to the above problem, we pass in Subsection \ref{proof3} to checking that operators $\mathbf{P}$, $\mathbf{Q}$ and $\mathbf{R}$ defined in Subsection \ref{proof1} indeed have got the required properties established for  $\mathcal{P}$, $\mathcal{Q}$ and $\mathcal{R}$.

To start our considerations, we recall the following generalized Kakutani fixed point theorem (\cite[Theorem~4]{bk}):
\begin{theorem}
Suppose that $X$ is a real Banach space and let $M\subset X$ be nonempty, convex and compact. Suppose also that the multivalued operator $T:M\longrightarrow 2^M$ satisfies the conditions:
\begin{enumerate}[a)]
\item $T$ has nonempty, closed and convex values (i.e $Tx$ has these properties for all $x\in M$),
\item $T$ has closed graph in $X\times X$.
\end{enumerate}
Then there exists an element $\bar{x}\in M$ such that $\bar{x}\in T\bar{x}$.
\label{thkakutani}\end{theorem}
\begin{flushright}$\Box$\end{flushright}
Before proceeding further, for a sequence $(f^i)_{i=1}^\infty$ of functions in $\Ce{\R^n}$ (denote $f^i=(f^i_1,\ldots,f^i_n)$) and a function $f\in\Ce{\R^n}$ (denote $f=(f_1,\ldots,f_n)$), we introduce following property which will be useful in the sequel:
\begin{eqnarray}
&&\textrm{for all $t\in (0,T),\ \delta>0$ there exists $i_0(t,\delta)\in\mathbb{N},\ h_0(\delta)\in(0,\delta]$ such that}\nonumber\\
&&\max_{1\leq k\leq n} \hnorm{f^i_k(t+h) - f_k(t)}<\delta\label{a6}\\
&&\textrm{for all $i\geq i_0(t,\delta), 0\leq\hnorm{h}\leq h_0(\delta), t+h\in(0,T)$} \,.\nonumber
\end{eqnarray}

Now, let $K$ be defined as in (\ref{000eqn1.1}) and $M_R$ be defined as in (\ref{000eqn1.2}). Suppose also that a multivalued operator $\mathcal{G}:M_R\rightarrow 2^{M_R}$ is of the form $\mathcal{G} = \mathcal{P}\circ\mathcal{Q}\circ\mathcal{R}$ where $\mathcal{P}$, $\mathcal{Q}$ and $\mathcal{P}$ satisfy:
\begin{assumption}
The operator $\mathcal{R}$ satisfies:
\begin{enumerate}[a)]
\item $\mathcal{R}:M_R\longrightarrow \Ce{\R^n}$,
\item \label{as1b}for any sequence $\kappa^i\rightarrow\kappa$ in $M_R$ with topology of $K$, the sequence $f^i=\mathcal{R}(\kappa^i)$ has a subsequence convergent to $f$ in sense of (\ref{a6}) with $f=\mathcal{R}(\kappa)$.
\end{enumerate}
\label{as1}\end{assumption}
\begin{assumption}
The multivalued operator $\mathcal{Q}$ satisfies:
\begin{enumerate}[a)]
\item \label{as2a}$\mathcal{Q}:\Ce{\R^n}\longrightarrow 2^{\eLp{\infty}{\R^m}}$,
\item \label{as2b}$\mathcal{Q}$ is closed in the following sense: if
\begin{enumerate}[i)]
\item $(f^i)_{i=1}^\infty$ in $\Ce{\R^n}$ converges to $f$ in sense of (\ref{a6}),
\item $v^i \stackrel{*}{\rightharpoonup} v$ in $\eLp{\infty}{\R^m}$,
\item $v^i\in\mathcal{Q}(f^i)$
\end{enumerate}
then $v\in\mathcal{Q}(f)$,
\item \label{as2c}$\mathcal{Q}$ is bounded, i.e. $\bigcup Im(\mathcal{Q})$ is a bounded subset of $\eLp{\infty}{\R^m}$, where $Im(\mathcal{Q}):=\mathcal{Q}\left(\Ce{\R^n}\right)$.
\end{enumerate}
\label{as2}\end{assumption}
\begin{assumption}
The operator $\mathcal{P}$ satisfies:
\begin{enumerate}[a)]
\item $\mathcal{P}:\eLp{\infty}{\R^m}\longrightarrow K$,
\item $\mathcal{P}|_{\bigcup Im(\mathcal{Q})}:\bigcup Im(\mathcal{Q})\longrightarrow M_R$,
\item \label{as3c}$\mathcal{P}|_{\bigcup Im(\mathcal{Q})}$ is closed in the following sense: if
\begin{enumerate}[i)]
\item $v^i \stackrel{*}{\rightharpoonup} v$ in $\bigcup Im(\mathcal{Q})$ with topology of $\eLp{\infty}{\R^q}$,
\item $\kappa^i\rightarrow \kappa$ in $M_R$ with topology of $K$,
\item $\kappa^i=\mathcal{P}|_{\bigcup Im(\mathcal{Q})}(v^i)$
\end{enumerate}
then $\kappa =\mathcal{P}|_{\bigcup Im(\mathcal{Q})}(v)$.
\end{enumerate}
\label{as3}\end{assumption}
The following lemma holds:
\begin{lemma} 
Under Hypothesis \ref{as1}, \ref{as2} and \ref{as3}, $\mathcal{G} = \mathcal{P}\circ\mathcal{Q}\circ\mathcal{R}:M_R\longrightarrow 2^{M_R}$ is a closed operator, with topology of $K$ used in $M_R$.
\label{le1}\end{lemma}
\begin{proof}
Suppose that $\kappa^i\rightarrow \kappa$ and $\xi^i\rightarrow\xi$ in $M_R$ with topology of $K$ and that $\xi^i\in\mathcal{G}(\kappa^i)$. The question remains open whether $\xi\in\mathcal{G}(\kappa)$.

It is equivalent to say that $\xi=\mathcal{P}(v)\textrm{ and }v\in\mathcal{Q}\circ\mathcal{R}(\kappa)$
for some $v\in\eLp{\infty}{\R^m}$. At the same time, $\xi^i\in\mathcal{G}(\kappa^i)$ is equivalent to $\xi^i=\mathcal{P}(v^i)\textrm{ and }v^i\in\mathcal{Q}\circ\mathcal{R}(\kappa^i)$
for some $v\in\eLp{\infty}{\R^m}$. Due to Hypothesis \ref{as2}, $v^i$ is a bounded sequence, thus there exists a weakly-star convergent subsequence (for simplicity, we keep the original indexes when passing to a subsequence):
\[v^i \stackrel{*}{\rightharpoonup} v \textrm{ in } \eLp{\infty}{\R^m} \,. \]
By Hypothesis \ref{as3}, we hence obtain $\xi=\mathcal{P}(v)$.

Now it remains to check that $v\in\mathcal{Q}\circ\mathcal{R}(\kappa)$, i.e. $v\in \mathcal{Q}(f) \textrm{ and } f\in \mathcal{R}(\kappa)$
for some $f\in\Ce{\R^m}$. At the same time, the condition $v^i\in\mathcal{Q}\circ\mathcal{R}(\kappa^i)$ is equivalent to $v^i\in \mathcal{Q}(f^i) \textrm{ and } f^i\in \mathcal{R}(\kappa^i)$
for some $f^i\in\Ce{\R^n}$. Due to Hypothesis \ref{as1} there exists an element $f$ and a subsequence of $f^i$ such that (again after reindexing)
\[ \textrm{$f^i\rightarrow f$ in sense of (\ref{a6})} \]
and moreover $f=\mathcal{R}(\kappa)$. By the above property, weak-star convergence of $v^i$ and Hypothesis \ref{as2}, we can claim that $v=\mathcal{Q}(f)$.

\pagebreak
To sum up, we infer that there are elements $f$ and $v$ such that $f=\mathcal{R}(\kappa)$, $v=\mathcal{Q}(f)$ and $\xi=\mathcal{P}(v)$, the latter completing the proof.
\end{proof}
The next lemma follows straight forward as the corollary of the former one:
\begin{lemma}
Under Hypothesis \ref{as1}, \ref{as2} and \ref{as3}, the operator $\mathcal{G} = \mathcal{P}\circ\mathcal{Q}\circ\mathcal{R}:M_R\longrightarrow 2^{M_R}$ has closed values.
\label{co1}\end{lemma}
\begin{flushright}$\Box$\end{flushright}

Aiming to fulfill the assumptions of Theorem~\ref{thkakutani}, we still need to guarantee that $\mathcal{G}$ has nonempty and convex values. We need to make another hypotheses which will assure these properties:
\begin{assumption}
$\mathcal{P}$ is an affine operator, i.e. for every $x\in\eLp{\infty}{\R^m}$ we have $\mathcal{P}(x)=y_0 + \mathcal{P}_0(x)$ where $y_0\in K$ and $\mathcal{P}_0:\eLp{\infty}{\R^m}\longrightarrow K$ is a linear operator.
\label{as5}\end{assumption}
\begin{assumption}
Suppose that a multivalued function $W:[0,T]\times \R^n \longrightarrow 2^{\R^m}$, $W=(W_1,\ldots,W_m)$ is given where $W_j$ satisfy the general assumption (\ref{asm3}). We assume that $\mathcal{Q}$ is defined in the following way: \\
for given $f\in \Ce{\R^n}$ any $v\in\eLp{\infty}{\R^m}$ satisfies $v\in\mathcal{Q}(f)$ if and only if $v(t)\in W(t,f(t))$ a.e. on $(0,T)$.
\label{as4}\end{assumption}
\begin{lemma}
Under Hypothesis \ref{as4}, the operator $\mathcal{Q}$:
\begin{enumerate}[a)]
\item has nonempty and convex values
\item is bounded, i.e. there is come $C>0$ such that $\mathcal{Q}(f)\subset B_{\eLp{\infty}{R^m}}(0,C)$ for every $f\in \Ce{\R^n}$
\end{enumerate}
\label{le2}\end{lemma}
\begin{proof}
By the boundedness of $W$, the second assertion follows directly. The convexity also is in turn a straightforward consequence of the convexity of the values of $W$.

It remains to prove that for arbitrary  $f\in\Ce{\R^n}$ there is some $v=(v_1,\ldots,v_m)\in\eLp{\infty}{\R^m}$ which belongs to $\mathcal{Q}(f)$, or, equivalently that $v_j(t)\in W_j(t,f(t))$ a.e. on $(0,T)$ for $j=1,\ldots,m$. By the boundedness of $W$, it is enough to check that a measurable selection is possible, namely that there exists some measurable $v_j$ satisfying the last condition. But this is assured by Theorem \ref{th4}.
\end{proof}
It follows from the above result that:
\begin{lemma}
Under Hypothesis \ref{as1}, \ref{as3}, \ref{as5} and \ref{as4}, the operator \\$\mathcal{G}=\mathcal{P}\circ\mathcal{Q}\circ\mathcal{R}:M_R\longrightarrow 2^{M_R}$ has nonempty and convex values.
\label{le3}\end{lemma}
\begin{proof}
The non-emptiness is a consequence of Lemma \ref{le2} and the well-posedness of the operators $\mathcal{P}$ and $\mathcal{R}$. 

Let us still check the convexity. Suppose that $\xi^1, \xi^2\in\mathcal{G}(\kappa)$ for some $\kappa\in M_R$. We want to show that for arbitrary $\lambda\in(0,1)$ there holds $\xi^\lambda:=\lambda\xi^1+(1-\lambda)\xi^2\in\mathcal{G}(\kappa)$. It is equivalent to say that $\xi^\lambda=\mathcal{P}(v^\lambda)\textrm{ and }v^\lambda\in\mathcal{Q}\circ\mathcal{R}(\kappa)$ for some $v^\lambda$. 

Note that the condition $\xi^i\in\mathcal{G}(\kappa)$ is equivalent to $\xi^i=\mathcal{P}(v^i)\textrm{ and }v^i\in\mathcal{Q}\circ\mathcal{R}(\kappa)$ 
for some $v^i$, $i=1,2$. Thus let $v^\lambda=\lambda v^1+(1-\lambda)v^2$. $v^\lambda\in\mathcal{Q}\circ\mathcal{R}(\kappa)$ because $\mathcal{Q}$ has convex values, as stated in Lemma \ref{le2}. Now, taking into account the affinity of $\mathcal{P}$, $\xi^\lambda=\mathcal{P}(v^\lambda)$ can be obtained by direct calculation.
\end{proof}
Due to the above lemmas, under Hypothesis \ref{as1} -- \ref{as4} it is possible to use Theorem~\ref{thkakutani}, still there is one more one simplification that can be made, namely it can be shown that Hypothesis \ref{as4} is stronger that Hypothesis \ref{as2}.
\begin{lemma}
If the operator $\mathcal{Q}$ satisfies Hypothesis \ref{as4}, then it also satisfies Hypothesis \ref{as2}.
\label{le4}\end{lemma}
\begin{proof}
Condition \ref{as2a}) in Hypothesis \ref{as2} follows directrly, condition \ref{as2c}) is assured by Lemma \ref{le2}. The proof that the condition \ref{as2b}) is assured can be found in \cite[theorem~3.6]{gs1} for the case $n=m$, an analogue of this assertion was used in \cite[lemma~4.4]{hns} for $n\neq m$.
\end{proof}
Altogether, as a concluding corollary we can state the final theorem of the above considerations:
\begin{theorem}
Suppose that $\mathcal{G} = \mathcal{P}\circ\mathcal{Q}\circ\mathcal{R}$ where $\mathcal{P}$, $\mathcal{Q}$, $\mathcal{R}$ satisfy Hypothesis \ref{as1}, \ref{as3}, \ref{as5} and \ref{as4}, respectively. Then there exists a fixed point of $\mathcal{G}$ in $M_R$, i.e. there exists a $\kappa\in M_R$ such that $\kappa\in \mathcal{G}(\kappa)$.
\label{th6}\label{cor6}\end{theorem}
\begin{proof}
We can apply Theorem~\ref{thkakutani} with $X=K$ and $M=M_R$. Its hypotheses are assured by Lemma \ref{le4}, Lemma \ref{le1} (closedness of the fixed-point operator), Lemma \ref{co1} (closedness of its values) and Lemma \ref{le3}.
\end{proof}
\begin{remark}The proofs of Lemma \ref{le1} and \ref{le3} exploited some arguments utilized in the proof of \cite[theorem~4.1]{hns}.
\end{remark}

\subsection{Conclusion of the proof of existence}\label{proof3}
This subsection is devoted to checking that the operators $\mathbf{P}$, $\mathbf{Q}$ and $\mathbf{R}$ from Subsection \ref{proof1} satisfy Hypothesis \ref{as1}, \ref{as3}, \ref{as5} and \ref{as4} stated in Subsection \ref{proof2}. As stated in Theorem \ref{th6}, this will imply that the superposition $\mathbf{P}\circ\mathbf{Q}\circ\mathbf{R}$ has at least one fixed point $\kappa\in M_S$ (where $M_R$ is defined in (\ref{000eqn1.2}) and $S$ is the constant set in assumption (\ref{asm4a})). Due to the Definition \ref{000def4} (equivalent to Definition \ref{000def3}), it will mean that the fixed point $\kappa$ determines a solution $(\mathfrak{u}(\kappa),\kappa)$ of the problem (\ref{eqnmain}) (where $\mathfrak{u}$ is the operator defined in (\ref{000eqn1.3})), hence completing the proof of Theorem \ref{thmain}. We remind that for Theorem \ref{thmain}, all assumptions (\ref{asm1}) - (\ref{asm7}) are needed.

Some of the properties specified in Hypothesis \ref{as1}, \ref{as3}, \ref{as5} and \ref{as4} have already got justified in Subsection \ref{proof1}. Moreover, it follows readily that $\mathbf{P}$ is affine. It remains to prove that $\mathbf{P}$ satisfies Hypothesis \ref{as3} \ref{as3c}) while $\mathbf{R}$ satisfies Hypothesis \ref{as1} \ref{as1b}).

\begin{lemma}
If $v^i \stackrel{*}{\rightharpoonup} v$ in $\eLp{\infty}{\R^m}$, $\kappa^i\longrightarrow \kappa$ in $K$ and $\kappa^i=\mathbf{P}(v^i)$ then $\kappa =\mathbf{P}(v)$.
\end{lemma}
The proof of the above lemma follows along the lines of part of the proof of \cite[theorem~4.1]{hns}.
\begin{flushright}$\Box$\end{flushright}

\begin{lemma}
If $\xi^i\rightarrow\xi$ in $K$, $\xi^i,\xi\in M_S$, then the sequence $f^i=\mathbf{R}(\xi^i)$ has a subsequence convergent in sense of (\ref{a6}) to $f=\mathbf{R}(\xi)$.
\end{lemma}
\begin{proof}
As in the proof of \cite[theorem~4.1]{hns}, we are going to show that the sequence $(f^i)_{i=1}^\infty$ has a subsequence $(f^{i_s})_{s=1}^\infty$ such that for every $k=1,\ldots,n$ for every $\delta_k>0$ and $t\in (0,T)$ the inequality $\hnorm{f_k^{i_s}(t+h)-f_k(t)}<\delta_k$ holds for $s$ large enough ($s\geq s_0(t,\delta_k)$) and $h$ small enough ($0\leq \hnorm{h}\leq h_0(\delta_k)$ with $h_0(\delta_k)\in(0,\delta_k]$).

Consider the sequence $(f^i_k)_{i=1}^\infty$ for some $k\in\{1,\ldots,n\}$, where $f^i$ are defined as in the statement of the lemma. By the definition of the operator $\mathbf{R}$ and by the triangle inequality,
\begin{eqnarray} 
\hnorm{f_k^i(t+h)-f_k(t)} &=& \hnorm{\mathfrak{u}(\xi^i)(x^*_k,t+h) - \mathfrak{u}(\xi)(x^*_k,t)}\nonumber\\
&\leq&\hnorm{\mathfrak{u}(\xi^i)(x^*_k,t+h) - \mathfrak{u}(\xi^i)(x^*_k,t)}\ +\ \hnorm{\mathfrak{u}(\xi^i)(x^*_k,t) - \mathfrak{u}(\xi^i)(\tilde{x},\tilde{t})}\nonumber\\
&+&\hnorm{\mathfrak{u}(\xi^i)(\tilde{x},\tilde{t}) - \mathfrak{u}(\xi)(\tilde{x},\tilde{t})}\ +\ \hnorm{\mathfrak{u}(\xi)(\tilde{x},\tilde{t}) - \mathfrak{u}(\xi)(x^*_k,t)}\label{000eqn3.1}
\end{eqnarray}
holds for arbitrary $(\tilde{x},\tilde{t})\in Q_T$. We will individually estimate each of the above four right hand side terms.

By Theorem \ref{th3}, the first of these terms can be estimated by:
\begin{eqnarray}
\hnorm{\mathfrak{u}(\xi^i)(x^*_k,t+h) - \mathfrak{u}(\xi^i)(x^*_k,t)} &\leq& c_6\,\left|\left|\left| (x^*_k,t+h) - (x^*_k,t) \right|\right|\right|^\alpha\nonumber\\
&= & c_6\,\left|\left|\left| (0,h) \right|\right|\right|^\alpha\ =\ c_6\,(h/4)^\alpha\ < \delta_k/4\label{000eqn3.2}
\end{eqnarray}
for $h<4\,(\delta_k/(4\,c_6))^{1/\alpha}$ (where $\left|\left|\left|\: .\: \right|\right|\right|$ is defined in the statement of Theorem \ref{th3}).

Theorem \ref{th3} allows also to infer analogous estimates for two of the remaining three terms in (\ref{000eqn3.1}):
\begin{eqnarray}
&&\hnorm{\mathfrak{u}(\xi^i)(x^*_k,t) - \mathfrak{u}(\xi^i)(\tilde{x},\tilde{t})}\ \leq\ c_6\,\left|\left|\left| (x^*_k - \tilde{x},t-\tilde{t}) \right|\right|\right|^\alpha\ <\  \delta_k/4\label{000eqn3.3}\\
&&\hnorm{\mathfrak{u}(\xi)(\tilde{x},\tilde{t}) - \mathfrak{u}(\xi)(x^*_k,t)}\ \leq\ c_6\,\left|\left|\left| (\tilde{x}-x^*_k,\tilde{t}-t) \right|\right|\right|^\alpha\ <\ \delta_k/4\label{000eqn3.4}
\end{eqnarray}
for $(\tilde{x},\tilde{t})$ close enough to $(x^*_k,t)$, i.e. for 
\[(\tilde{x},\tilde{t})\in B:=\left\{(x,t):\, \left|\left|\left| (x-x^*_k,t-t) \right|\right|\right| < (\delta_k/(4\,c_6))^{1/\alpha} \right\} \,.\]

Now we are left to estimate the remaining term in (\ref{000eqn3.1}), namely $\hnorm{\mathfrak{u}(\xi^i)(\tilde{x},\tilde{t}) - \mathfrak{u}(\xi)(\tilde{x},\tilde{t})}$. To this purpose, we will show that the sequence $\left(\mathfrak{u}(\xi^i)(\tilde{x},\tilde{t})\right)_{i=1}^{\infty}$ has a subsequence, denote it $\left(\mathfrak{u}(\xi^{i_s})(\tilde{x},\tilde{t})\right)_{s=1}^{\infty}$, such that, for some $(\tilde{x},\tilde{t})\in B$,
\begin{equation}
\mathfrak{u}(\xi^{i_s})(\tilde{x},\tilde{t}) \rightarrow \mathfrak{u}(\xi)(\tilde{x},\tilde{t})\quad\textrm{as $s\rightarrow\infty$} \,.
\label{000eqn3.5}\end{equation}
By Theorem \ref{th2} and assumption (\ref{asm5b}):
\begin{eqnarray*}
\xnorm{\mathfrak{u}(\xi^i) - \mathfrak{u}(\xi)}{L^1(Q_T)}&\leq& c_5\xnorm{g(\,.\,,\,.\,;\xi^i(\,.\,)) - g(\,.\,,\,.\,;\xi(\,.\,))}{L^1(Q_T)}\\
&\leq& c_5c_3 \sum_{j=1}^{m}\xnorm{\xi^i_j - \xi_j}{C[0,T]}\,,
\end{eqnarray*}
the latter right-hand side is convergent to $0$ by assumptions of the lemma, hence there exists a subsequence $\left(\mathfrak{u}(\xi^{i_s})\right)_{s=1}^{\infty}$ convergent to $\mathfrak{u}(\xi)$ a.e. on $Q_T$. In particular, this subsequence is convergent a.e. on $B$ and hence we can choose $(\tilde{x},\tilde{t})\in B$ such that the convergence (\ref{000eqn3.5}) holds, i.e. 
\begin{equation}
\hnorm{\mathfrak{u}(\xi^{i_s})(\tilde{x},\tilde{t}) - \mathfrak{u}(\xi)(\tilde{x},\tilde{t})}\ <\ \delta_k/4
\label{000eqn3.6}\end{equation}
for $s>s_0(\delta_k)$ and, besides, (\ref{000eqn3.3}), (\ref{000eqn3.4}) are satisfied. Thus, considering the subsequence $\left(\mathfrak{u}(\xi^{i_s})\right)_{s=1}^{\infty}$ and the point $(\tilde{x},\tilde{t})$ as above and taking (\ref{000eqn3.1}), (\ref{000eqn3.2}), (\ref{000eqn3.3}), (\ref{000eqn3.4}) and (\ref{000eqn3.6}) into account, we obtain
\[ \hnorm{f_k^{i_s}(t+h)-f_k(t)}\ <\ \delta_k\]
for $s>s_0(\delta_k)$ and $h<4\,(\delta_k/(4\,c_6))^{1/\alpha}$ what concludes the proof. 
\end{proof}

To sum up, $M_S$ is a nonempty, convex and compact subset of the Banach space $K$ of admissible controls, the operators  $\mathbf{P}$, $\mathbf{Q}$ and $\mathbf{R}$ fulfill Hypothesis \ref{as1}, \ref{as3}, \ref{as5} and \ref{as4}, thus due to Theorem~\ref{th6} we can deduce that the system (\ref{eqnmain}) has at least one solution in sense of Definition \ref{000def3} (or, equivalently, Definition \ref{000def4}) given by $(\mathfrak{u}(\kappa),\kappa)$, where $\kappa$ is a fixed point of $\mathbf{P}\circ\mathbf{Q}\circ\mathbf{R}$.

\section{Appendix}\label{appendix}\setcounter{equation}{0}
Below, technical definitions and results utilized in other sections of present paper will be provided, in part originating from other publications (\cite{as}, \cite{gs1}, \cite{np}, \cite{rob}). Still some of the formulations are adapted to the context of our paper, thus require more comments.
\begin{definition} For two arbitrary Banach spaces $X$ and $Y$, a multivalued function $F:X\longrightarrow 2^Y$ is said to be upper semicontinuous if and only if, for every $x\in X$ and for every open set $O\subset Y$ such that $F(x)\subset O$, there exists a neighborhood $U(x)$ of $x$ such that $F(U(x))\subset O$.
\label{000defap1}\end{definition}
\begin{remark}Note that for singlevalued functions the above definition reduces to the definition of continuity, not just semicontinuity. Thus, the introduced property is stronger than the standard semicontinuity.
\end{remark}
\begin{definition}
For two arbitrary Banach spaces $X$ and $Y$, a multivalued operator $T:X\longrightarrow 2^Y$ is said to be closed if and only if, whenever $x_n\rightarrow x$ in $X$, $y_n\rightarrow y$ in $Y$, $y_n\in T(x_n)$ then $y\in T(x)$. Equivalently, we can say that the graph of $T$ is closed in $X\times Y$.
\label{000defap2}\end{definition}
\begin{definition}
For a Banach space $X$ and a multivalued operator $T:X\longrightarrow 2^X$, an element $\bar{x}\in X$ is said to be a fixed point of $T$ if and only if $\bar{x}\in T(\bar{x})$.
\label{000defap3}\end{definition}
\begin{theorem}
Suppose that assumptions (\ref{asm1}), (\ref{asm2a}), (\ref{asm2b}), (\ref{asm5a}), (\ref{asm6}) are fulfilled. Then, for arbitrary $\kappa\in K$ the weak solution $u$ of problem (\ref{000eqn1}) in sense of Definition \ref{000def1} exists, is unique and belongs to $L^\infty(Q_T)$.
\label{th1}\end{theorem}
\begin{proof}
The proof of the existence and uniqueness can be done with the methods similar to those in the proof of \cite[Theorem~8.4]{rob} (still, the function spaces and restrictions for $f$ there considered are slightly different from the ones assumed in \cite{rob} what should be taken into account when showing the suitable estimates for the proof of Theorem \ref{th1}). A formal idea of the boundedness proof consists in showing for a.e. $x\in\Omega$ that $u(x,t)$ is uniformly bounded in $L^{\infty}(0,T)$. To prove the estimate for given $x\in\Omega$, we test the main equation of (\ref{000eqn1}) with $u$ multiplied by a smooth cut-off function $\xi_{\epsilon,h}$ such that
\begin{displaymath}
\begin{array}{ll} 
\xi_{\epsilon,h}\equiv 1& \textrm{on}\ B(x,\epsilon)\times (0,T) \,,\\
\xi_{\epsilon,h}\equiv 0& \textrm{on}\ (B(x,\epsilon+h)\times (0,T))^c \,,\\
\hnorm{\nabla\xi_{\epsilon,h}}\equiv h^{-1}& \textrm{on}\ \left(B(x,\epsilon+h)\setminus B(x,\epsilon)\right)\times (0,T) \,.
\end{array}
\end{displaymath}
In the obtained integral identity we pass with $h\rightarrow 0$ and then, using the Lebesgue differentiation theorem, with $\epsilon\rightarrow 0$. After that, the required estimate can be concluded by the Gronwall inequality.
\end{proof}
\begin{theorem}
Suppose that the assumptions of Theorem \ref{th1} are satisfied and $u_1$ and $u_2$ are solutions of the system (\ref{000eqn1}) corresponding to controls $\kappa_1$ and $\kappa_2$, respectively. Then there exists a positive constant $c_5$, independent on the control and initial conditions, such that:
\begin{displaymath}
\xnorm{u_1-u_2}{L^1(Q_T)}+\xnorm{u_1-u_2}{L^2(Q_T)}^2 \leq c_5\,\xnorm{g(\,.\,,\,.\,;\kappa_1(\,.\,))-g(\,.\,,\,.\,;\kappa_2(\,.\,))}{L^1(Q_T)} \,.
\end{displaymath}
\label{th2}\end{theorem}
\begin{proof}
After transforming the integral identity in Definition \ref{000def1} to the form with all derivatives put on a test function, the proof can be performed with arguments analogous to those used in the proof of \cite[Theorem~4.3]{np}.
\end{proof}
\begin{theorem}
Suppose that assumptions of Theorem \ref{th1} and assumptions (\ref{asm2c}), (\ref{asm4b}) are satisfied. Denote also
\[ \left|\left|\left| (x,t) \right|\right|\right| = \max\left(x_1^2,\ldots,x_d^2,\hnorm{\frac{t}{4}}\right)\]
for every $(x,t)\in\R^{d+1}$. Then:
\begin{enumerate}[a)]
\item\label{th3a} For arbitrary $\kappa\in K$ there exist constants $\alpha\in (0,1)$ and $c_6 >0$ such that
\[ \hnorm{u(x_1,t_1) - u(x_2,t_2)} \leq c_6 \left|\left|\left| (x_1,t_1)-(x_2,t_2) \right|\right|\right|^\alpha\]
for all $(x_1,t_1),(x_2,t_2)\in \textrm{int}\,Q_T$ (note that we exclude the choice of $(x_1,t_1),(x_2,t_2)\in \partial Q_T$).
\item\label{th3b} The constant $\alpha$ is not dependent on $\kappa$ and the constant $c_6$ depends on $\xnorm{g(\,.\,,\,.\,;\kappa(\,.\,))}{\eLp{q}{\Lp{p}}}$ continuously.
\item\label{th3c} If, additionally, assumptions (\ref{asm4a}) and (\ref{asm5c}) are fulfilled and the choice of $\kappa$ is restricted to the set $M_S$, then the constant $c_6$ can be chosen such that is not dependent on $\kappa$, i.e. $c_6 = c_6(\kappa)$ is a constant function on $M_S$.
\end{enumerate}
\label{th3}\end{theorem}
\begin{proof}
A proof of \ref{th3a}) can be found in \cite[Theorem~4]{as}, however we still give here short consistency remarks. 

The definition of weak solutions to a parabolic equation considered in \cite{as} (whose particular case is (\ref{000eqn1})) differs from Definition \ref{000def1} (the former assumes that the time derivative is put on a test function and that the test functions are of class $C^{1}(Q_T)\cap\{\phi(T,\,.\,)\equiv 0\}$), hence first we need to transform the integral identity in Definition \ref{000def1} to the form as in \cite{as}. 

Besides, in \cite{as} the nonlinear term is assumed to satisfy $\hnorm{\bar{f}(s)}\leq c_7\hnorm{s}$ with respect to the values of solution, differing from our case (since assumption (\ref{asm2}) allows e.g. $f(s)=-s^3+s$). But due to Theorem \ref{th1} for a given initial condition $u_0$ and a control $\kappa$ any solution $u$ of (\ref{000eqn1}) is bounded. It can be shown that a bounded solution $u$ of (\ref{000eqn1}) is also a solution of an analogous system with the nonlinear $f$ term replaced by
\[ \bar{f}(s)=\left\{\begin{array}{ll} 
+c_7&\textrm{for $s$ s.t. $s>c_7$}\\
f(s)&\textrm{for $s$ s.t. $s\in[-c_7,c_7]$}\\
-c_7&\textrm{for $s$ s.t. $s<-c_7$}
\end{array}\right. \,,\]
where $c_7=\textrm{ess sup}_{(x,t)\in Q_T}\hnorm{u(x,t)}$. $\bar{f}$ as above belongs to the class of nonlinear terms admissible in \cite{as}.

After having overcome the above difficulties, \cite[Theorem~4]{as} can be applied.

The assertion \ref{th3b}) can be deduced by a technically extensive analysis of the proof of \cite[Theorem~4]{as}. The latter proof consists in using a Harnack-type inequality, shown by suitable cut-off techniques.

The assertion \ref{th3c}) is a consequence of \ref{th3b}). Since $c_6$ depends on $\xnorm{g(\,.\,,\,.\,;\kappa(\,.\,))}{\eLp{q}{\Lp{p}}}$ continuously and the latter value is bounded with respect to $\kappa$ for $\kappa\in M_S$, the values of $c_6$ also are bounded and thus it suffices to choose a constant $\bar{c_6}=\sup_{\kappa\in M_S}{c_6}$ in \ref{th3a}).
\end{proof}
\begin{theorem}
Let $W$ be a multivalued function satisfying assumptions (\ref{asm3a}), (\ref{asm3b}) and (\ref{asm3c}), $f\in C((0,T);\R^n)$ (where $C((0,T);\R^n)$ is considered with standard supremum norm) and $\mathbb{W}(t):=W(t,f(t))$ on $[0,T]$. Then there exists a measurable selection, i.e. a measurable function $\mathbb{V}:[0,T]\longrightarrow R$, such that $\mathbb{V}(t)\in \mathbb{W}(t)$ for a.e. $t\in [0,T]$.
\label{th4}\end{theorem}
\begin{flushright}$\Box$\end{flushright}
The proof of Theorem \ref{th4} follows in the same manner as that of \cite[Lemma~3.4]{gs1}.
\begin{theorem}
If $W$, $f$ and $\mathbb{W}$ are as in Theorem \ref{th4}, then there exists a solution to (\ref{000eqn2}) in the sense of Definition \ref{000def2}.
\label{th5}\end{theorem}
\begin{proof}
This is a direct corollary of Theorem \ref{th4}. Since $W$ obeys assumption (\ref{asm3c}), $\mathbb{W}$ is bounded and hence any selection $\mathbb{V}(t)\in \mathbb{W}(t)$ (for a.e. $t\in [0,T]$) is bounded as well. By Theorem \ref{th4}, there exists a measurable selection. A selection $\mathbb{V}\in L^\infty(0,T)$ which is bounded and measurable belongs to $L^{\infty}(0,T)$. The existence of a solution to (\ref{000eqn3}) in the Carath\'eodory sense is known (cf. \cite[Chap.~3, Sec.~3]{cl}).
\end{proof}

\section{Concluding remarks}\setcounter{equation}{0}
The methods introduced in this paper can be applied to a broad class of reaction-diffusion models, which share certain key properties with our model, including a whole range of Allen-Cahn models. These properties are: existence of solutions, their uniqueness and regularity and the system stability with respect to perturbations of a control term, as in corresponding theorems in Section \ref{appendix}.

Once the existence question is answered for (\ref{eqnmain}), one can pose a problem of optimal localization of control devices and measurement points. This is a direction of our present research. In the case of Lipschitz continuous $W_j$ in (\ref{eqnmain}) we expect to obtain a result which says that for a given reference state of the system and a given number of the control devices and measurement point it is possible to localize those devices and points in a way which is optimal with respect to the task of preserving the evolution of the system as close as possible to the reference state. Considering the latter restriction on $W_j$ also allows us to formulate the uniqueness result complementing the existence theorem presented in this paper. 

An optimal implementation of the control algorithms will be included in \textsc{Grzegorz Dudziuk}'s Ph.D. thesis (under preparation).

\end{document}